\documentclass[times,11pt]{elsarticle}
\usepackage{amssymb,amsmath,amsthm}
\newtheorem{theorem}{Theorem}[section]
\newtheorem{propo}[theorem]{Proposition}
\newtheorem{lem}[theorem]{Lemma}

\newtheorem{coro}[theorem]{Corollary}
\newtheorem{define}[theorem]{Definition}

\newcommand{\F}{\mathbb F}

\begin{document}

\begin{frontmatter}

\title{A note on primitive $1-$normal elements over finite fields}
\author{Lucas Reis}
\ead{lucasreismat@gmail.com}
\address{Departamento de Matem\'{a}tica, Universidade Federal de Minas Gerais, UFMG, Belo Horizonte, MG, 30123-970, Brazil}
\begin{abstract}
Let $q$ be a prime power of a prime $p$, $n$ a positive integer and $\F_{q^n}$ the finite field with $q^n$ elements. The $k-$normal elements over finite fields were introduced and characterized by Huczynska {\it et al} (2013). Under the condition that $n$ is not divisible by $p$, they obtained an existence result on primitive $1-$normal elements of $\F_{q^n}$ over $\F_q$ for $q>2$. In this note, we extend their result to the excluded case $q=2$.
\end{abstract}
\begin{keyword}
{Finite Fields, Normal Basis, k-normal elements, Primitive elements}

2010 MSC: 12E20 \sep 11T06
\end{keyword}
\end{frontmatter}

\section{Introduction}
Let $\F_{q^n}$ be the finite field with $q^n$ elements, where $q$ is a prime power and $n$ is a positive integer. Recall that an element $\alpha\in \F_{q^n}$ is said to be normal over $\F_q$ if $A=\{\alpha, \alpha^q, \cdots, \alpha^{q^{n-1}}\}$ is a basis of $\F_{q^n}$ over $\F_q$; A is called a normal basis. Normal basis are frequently used in cryptography and computer algebra systems; sometimes it is useful to take normal basis composed by primitive elements, i.e., generators of the multiplicative group $\F_{q^n}^*$. The {\it Primitive Normal Basis Theorem} states that for any extension field $\F_{q^n}$ of $\F_q$, there exists a basis composed by primitive normal elements; this result was first proved by Lenstra and Schoof \cite{lenstra} and a proof without the use of a computer was latter given in \cite{cohen}. 

A characterization of normal elements is given in (\cite{lidl}, Theorem 2.39): an element $\alpha\in \F_{q^n}$ is normal over $\F_q$ if and only if the polynomials $$g_\alpha(x):= \sum_{i=0}^{n-1}\alpha^{q^i}x^{n-1-i}\quad\text{and}\quad x^n-1,$$ are relatively prime over $\F_{q^n}$. Motivated by this characterization, in \cite{panario}, the authors introduce $k-$normal elements:
\begin{define}
Let $\alpha\in \F_{q^n}^*$ and $g_{\alpha}(x)=\sum_{i=0}^{n-1}\alpha^{q^i}\cdot x^{n-1-i}$. We say that $\alpha$ is $k-$normal over $\F_q$ if the greatest common divisor of $x^n-1$ and $g_{\alpha}(x)$ over $\F_{q^n}$ has degree $k$.
\end{define}
From definition, $0-$normal elements correspond to normal elements in the usual sense. In the same paper, the authors give a characterization of $k-$normal elements and find a formula for their number. Also, they obtain an existence result on primtive $1-$normal elements:

\begin{theorem}[\cite{panario}, Theorem 5.10] \label{mainpanario}
Let $q=p^e$ be a prime power and $n$ a positive integer not divisible by $p$. Assume that $n\ge 6$ if $q\ge 11$ and that $n\ge 3$ if $3\le q\le 9$. Then there exists a primitive $1-$normal element of $\F_{q^n}$ over $\F_q$.
\end{theorem}
The authors propose an extension of the above theorem for all pairs $(q, n)$ with $n\ge 2$ as a problem (\cite{panario}, Problem 6.2); they conjectured that such elements always exist. However, it was proved in \cite{alizadeh} that for odd $q>3$ and $n=2$, there are no primitive $1-$normal elements of $\F_{q^2}$ over $\F_q$.  The aim of this note is to extend Theorem \ref{mainpanario} to the case when $q=2$ and $n$ is odd. Essentially, we show that the tools used in \cite{panario} to prove Theorem \ref{mainpanario} can be adapted to that case. 

\section{Existence of primitive $1-$normal elements over $\F_2$}
First, we present some definitions and results that will be useful in the rest of this paper. 

\begin{define}
\hfill \break
\begin{enumerate}[(a)]
\item Let $f(x)$ be a monic polynomial with coefficients in $\F_q$. The Euler Phi Function for polynomials over $\F_q$ is given by $$\Phi_q(f)=\left |\left(\frac{\F_q[x]}{\langle f\rangle}\right)^{*}\right |,$$ where $\langle f\rangle$ is the ideal generated by $f(x)$ in $\F_q[x]$. 
\item If $t$ is a positive integer (or a monic polynomial over $\F_q$), $W(t)$ denotes the number of square-free (monic) divisors of $t$.
\end{enumerate}
\end{define}

We have an interesting formula for the number of $k-$normal elements over finite fields:

\begin{lem}[\cite{panario}, Theorem 3.5] \label{count} The number $N_k$ of $k-$normal elements of $\F_{q^n}$ over $\F_q$ is given by
\begin{equation}\label{eq2}\sum_{h|x^n-1\atop{\deg(h)=n-k}}\Phi_q(h),\end{equation}
where the divisors are monic and polynomial division is over $\F_q$.
\end{lem}

In particular, if $n\ge 2$, the number of $1-$normal elements of $\F_{q^n}$ over $\F_q$ is at least equal to $\Phi_q(T)$, where $T=\frac{x^n-1}{x-1}$.

\subsection{A sieve inequality}
The proof of Theorem \ref{mainpanario} is based in an application of the {\it Lenstra-Schoof} method, introduced in \cite{lenstra}; this method has been used frequently in the characterization of elements in finite fields with particular properties like being primitive, normal and of zero-trace. For more details, see \cite{cohen} and \cite{panario}. In particular, from Corollary 5.8 of \cite{panario}, we can easily deduce the following:

\begin{lem}\label{aux2}
Suppose that $q$ is a power of a prime $p$, $n\ge 2$ is a positive integer not divisible by $p$ and $T(x)=\frac{x^n-1}{x-1}$. If
\begin{equation}\label{sieve}W(T)\cdot W(q^n-1)< q^{n/2-1}, \end{equation}
then there exist $1-$normal elements of $\F_{q^n}$ over $\F_q$.
\end{lem}

Inequality \eqref{sieve} is an essential step in the proof of Theorem \ref{mainpanario} and it was first studied in \cite{cohen2}; under the condition that $n\ge 6$ for $q\ge 11$ and $n\ge 3$ for $3\le q\le 9$, this inequality is not true only for a finite number of pairs $(q, n)$ (see Theorem 4.5 of \cite{cohen2}). Here we extend the study of inequality \eqref{sieve} to the case when $q=2$ and $n$ is odd. First, we have the following:

\begin{propo}\label{aux}
Suppose that $n\ge 3$ is odd and $T(x):=\frac{x^n-1}{x-1}\in \F_2[x]$. Then $W(T)\le 2^{\frac{n+9}{5}}.$
\end{propo}

\begin{proof}
For each $2\le i\le 4$, let $s_i$ be the number of irreducible factors of degree $i$ dividing $T(x)$. Since $n$ is odd, $T(x)$ has no linear factor. By a direct verification we see that the number of irreducible polynomials over $\F_2$ of degrees $2, 3$ and $4$ is $1, 2$ and $3$, respectively. Hence $s_2\le 1, s_3\le 2$ and $s_4\le 3$. In particular, the number of irreducible factors of $T(x)$ over $\F_2$ is at most $$\frac{n-1-2s_2-3s_3-4s_4}{5}+s_2+s_3+s_4=\frac{n-1+3s_2+2s_3+s_4}{5}.$$Since $\frac{n-1+3s_2+2s_3+s_4}{5}\le \frac{n-1+3+4+3}{5}=\frac{n+9}{5}$, we conclude the proof.
\end{proof}

According to Lemma 7.5 in \cite{cohen3}, $W(2^n-1)< 2^{\frac{n}{7}+2}$ if $n$ is odd. In particular, we obtain the following:

\begin{coro}\label{aux3}
Suppose that $n\ne 15$ is odd, $q=2$ and $T(x)=\frac{x^n-1}{x-1}\in \F_2[x]$. For $n>9$, inequality \eqref{sieve} holds.
\end{coro}
\begin{proof}
Notice that $\frac{n+9}{5}+\frac{n}{7}+2< \frac{n}{2}-1$ for $n\ge 31$. From Proposition \ref{aux} and Lemma 7.5 of \cite{cohen3}, it follows that inequality \eqref{sieve} holds for odd $n\ge 31$. The remaining cases can be verified directly. 
\end{proof}

We are ready to state and prove our result:

\begin{theorem}
Suppose that $n\ge 3$ is odd. Then there exist a primitive $1-$normal element of $\F_{2^n}$ over $\F_2$.
\end{theorem}

\begin{proof}
According to Lemma \ref{aux2} and Corollary \ref{aux3}, this statement is true for $n> 9$  if $n\ne 15$. For the remaining cases $n=3, 5, 7, 9$ and $15$ we use the following argument. Let $P$ be the number of primitive elements of $\F_{2^n}$ and $N_1$ the number of $1-$normal elements of $\F_{2^n}$ over $\F_2$; if $P+N_1> 2^n$, there exists a primitive $1-$normal element of $\F_{2^n}$ over $\F_2$. Notice that $P=\varphi(2^n-1)$ and, according to Lemma \ref{count}, $N_1\ge \Phi_2\left(\frac{x^n-1}{x-1}\right)$. By a direct calculation we see that 
$$\varphi(2^n-1)+\Phi_2\left (\frac{x^n-1}{x-1}\right)>2^n,$$
for $n=3, 5, 7, 9$ and $15$. This completes the proof.
\end{proof}


\begin{thebibliography}{99}
\bibitem{alizadeh} M. Alizadeh. {\it Some notes on the $k-$normal elements and $k-$normal polynomials over finite fields}, Journal of Algebra and Its Applications 16 (2017).

\bibitem{cohen2} S. D. Cohen, D. Hachenberger, {\it Primitive normal bases with prescribed trace}, Applicable Algebra in Engineering, Communication and Computing 9 (1999) 383–403.

\bibitem{cohen} S. D. Cohen, S. Huczynska. {\it The primitive normal basis theorem - without a computer}, Journal of the London Mathematical Society 67 (2003) 41-56.

\bibitem{cohen3} S. D. Cohen. {\it Pairs of primitive elements in fields of even order}, Finite Fields Appl. 28 (2014) 22-42.

\bibitem{panario} S. Huczynska, G.L. Mullen, D. Panario, and D. Thomson, {\it Existence and properties of $k-$normal elements over finite fields}, Finite Fields Appl. 24 (2013) 170-183.

\bibitem{lenstra} H. W. Lenstra, R. Schoof, {\it Primitive normal bases for finite fields}, Mathematics of Computation 48 (1987) 217-231.

\bibitem{lidl} R. Lidl, H. Niederreiter, {\it Finite Fields: Encyclopedia of Mathematics and Its Applications}, vol. 20, 2nd ed. Cambridge University Pres, Cambridge, 1997.

\end{thebibliography}
\end{document}